\documentclass[11pt,fleqn,a4paper]{article}

\usepackage{amsmath,amssymb,amsthm}
\usepackage[mathscr]{eucal}

\usepackage{hyperref}
\hypersetup{colorlinks, linkcolor=blue, citecolor=blue, urlcolor=blue}
\newcommand{\p}{\partial}
\newcommand{\rank}{\mathop{\rm rank}\nolimits}

\newtheorem{theorem}{Theorem}
\newtheorem{lemma}[theorem]{Lemma}

\newtheorem*{proposition*}{Proposition}
{\theoremstyle{definition}

\newtheorem{remark}[theorem]{Remark}

}

\newcommand{\todo}[1][\null]{\ensuremath{\clubsuit}}

\newcommand{\noprint}[1]{}

\begin{document}
\par\noindent {\LARGE\bf
Point and contact equivalence groupoids \\ 
of two-dimensional quasilinear hyperbolic equations
\par}

\vspace{4mm}\par\noindent {\large Roman O.\ Popovych
} \par\vspace{2mm}\par

\vspace{2mm}\par\noindent {\it
Fakult\"at f\"ur Mathematik, Universit\"at Wien, Oskar-Morgenstern-Platz 1, 1090 Wien, Austria
\par\noindent 
Institute of Mathematics of NAS of Ukraine, 3 Tereshchenkivska Str., 01024 Kyiv, Ukraine
}

\vspace{2mm}\par\noindent
E-mail: rop@imath.kiev.ua

\vspace{8mm}\par\noindent\hspace*{10mm}\parbox{140mm}{\small
We describe the point and contact equivalence groupoids 
of an important class of two-dimensional quasilinear hyperbolic equations. 
In particular, we prove that this class is normalized in the usual sense with respect to point transformations, 
and its contact equivalence groupoid is generated by the first-order prolongation 
of its point equivalence groupoid, the contact vertex group of the wave equation 
and a family of contact admissible transformations between trivially Darboux-integrable equations.
}\par\vspace{4mm}

\noprint{


}

\section{Introduction}\label{sec:Introduction}

Genuine (first-order) contact transformations between differential equations~\cite{olve1995A,stor2000A}, 
are rarer objects than their point counterparts. 
These transformations were introduced by Sophus Lie. 
In particular, it was known to him~\cite{kush2010a} that the class of Monge--Amp\'ere equations is closed 
under the action of its contact equivalence (pseudo)group, 
which consists of all contact transformations of two independent and one dependent variables. 
In the modern terminology~\cite{opan2020d,popo2010a,vane2020b}, 
this means that this class is normalized in the usual sense with respect to contact transformations. 
The well-known B\"acklund theorem~\cite[Theorem~4.32]{olve1995A} states that 
genuine contact transformations, which are not the first-order prolongations of point transformations, 
exist only in the case of one dependent variable.  

In the present paper, we compute the point and contact equivalence groupoids 
of the class~$\mathcal H_{\rm gen}$ of two-dimensional quasilinear hyperbolic equations of the form
\begin{gather}\label{eq:GenKGEqs}
u_{xy}=f(x,y,u,u_x,u_y),
\end{gather}
which is properly contained in the class of Monge--Amp\'ere equations. 
Here $x$ and~$y$ are the independent variables, 
$u$ is the unknown function of~$(x,y)$, which is treated as the dependent variable, 
and the arbitrary element~$f$ of the class runs through the set~$\mathcal A$ of differential functions for these variables
that depend at most on $(x,y,u,u_x,u_y)$, i.e., whose order is less than or equal to one. 
Subscripts of functions denote derivatives with respect to the corresponding variables,
e.g., $u_{xy}:=\p^2 u/\p x\,\p y$ and $f_{u_x}:=\p f/\p u_x$,
and we use the same notation for the derivatives $u_x$, $u_y$, $u_{xx}$, $u_{xy}$ and $u_{yy}$ 
and for the corresponding coordinates of 
the second-order jet space~$\mathrm J^2(\mathbb R^2_{x,y},\mathbb R^{}_u)$ with the independent variables $(x,y)$ and the dependent variable~$u$. 
The consideration is within the local framework. 
We employ the terms ``group'' and ``groupoid'' for pseudogroups and pseudogroupoids, respectively.
For a fixed value of the arbitrary element~$f$,
let $\mathcal E_f$, $G^{\rm c}_f$ and~$\mathcal G^{\rm c}_f$ denote the equation from the class~$\mathcal H_{\rm gen}$ with this value of~$f$,  
its contact symmetry group and its contact vertex group, $\mathcal G^{\rm c}_f:=\{(f,\Phi,f)\mid\Phi\in G^{\rm c}_f\}$.
See~\cite{vane2020b} for necessary notions of the theory of admissible transformations within classes of differential equations. 
The space with coordinates $(x,y,u,u_x,u_y)$ 
is the minimal underlying space for both the point and the contact equivalence groups of the class~$\mathcal H_{\rm gen}$, 
cf.\ \cite[footnote~1]{kuru2020a} and \cite[footnote~1]{opan2020d}. 
By $\mathrm D_x$ and~$\mathrm D_y$ we denote 
the operators of total derivatives with respect to~$x$ and~$y$, respectively,
$\mathrm D_x=\p_x+u_x\p_u+u_{xx}\p_{u_x}+u_{xy}\p_{u_y}+\cdots$ and
$\mathrm D_y=\p_y+u_y\p_u+u_{xy}\p_{u_x}+u_{yy}\p_{u_y}+\cdots$. 
\looseness=-1

The form~\eqref{eq:GenKGEqs}  is the most general form for equations that can be interpreted  
as (1+1)-dimensional generalized nonlinear Klein--Gordon equations in the light-cone (characteristic) coordinates. 
It is also the canonical form for single second-order two-dimensional partial differential equations 
each of whose two Monge systems admits a first integral of the first order~\cite[Theorem 12.1.2]{stor2000A}.
The class~$\mathcal H_{\rm gen}$ contains a number of famous equations, 
including the wave equation, the Klein--Gordon equation, the Liouville equation, the Tzitzeica equation 
and the sine-Gordon (or Bonnet) equation.
The problem of studying the subclass of equations of the form~\eqref{eq:GenKGEqs} with $f_{u_x}=f_{u_y}=0$
within the framework of group analysis of differential equations was posed by Sophus Lie~\cite{lie1881b}. 
The group classification problem for this subclass was completely solved in~\cite{boyk2021a}; 
see also a review of results on this and other subclasses of~$\mathcal H_{\rm gen}$ therein. 
In~\cite{lahn2005a,lahn2006a}, Lie symmetries had been classified for 
the wider subclass of~$\mathcal H_{\rm gen}$ that consists of the equations of the form~\eqref{eq:GenKGEqs} with $f_{u_x}=f_{u_y}$, 
which had been originally presented in the regular space-time variables, 
and the part of this classification for the above narrower subclass was enhanced in~\cite{boyk2021a}; 
see the discussion in the conclusions of~\cite{boyk2021a,vane2020b}.
The equations from the entire class~$\mathcal H_{\rm gen}$ 
for which each of the two Monge systems admits at least two functionally independent first integrals
are called hyperbolic Goutsat equations. 
Such equations were classified by Goutsat with respect to the equivalence generated 
by the point equivalence group~$G^\sim_{\rm gen}$ of the class~$\mathcal H_{\rm gen}$ 
(this group is given in Theorem~\ref{thm:GenKGEqsEquivGroupoid} below). 
See~\cite[Chapter~13]{stor2000A} for the presentation of the enhanced classification proof by Vessiot
and~\cite{kuzn2012a} for references on modern results related to Darboux-integrable equations. 

We describe contact admissible transformations within the class~$\mathcal H_{\rm gen}$, 
which essentially generalizes Lie's study in~\cite{lie1881b}. 
After deriving determining equations for such transformations 
in Section~\ref{sec:GenKGEqsContactAdmTransDetEqs}, 
we compute point and contact equivalence groupoids of the class~$\mathcal H_{\rm gen}$ in Section~\ref{sec:GenKGEqsEquivGroupoids}.
In particular, we prove that both the sources and targets of genuine contact admissible transformations 
are trivially Darboux-integrable equations.

\section{Determining equations for contact admissible transformations}
\label{sec:GenKGEqsContactAdmTransDetEqs}

Applying the direct method of finding admissible transformations,  
we fix a contact admissible transformation $\mathcal T=(f,\Phi,\tilde f)$ in the class~$\mathcal H_{\rm gen}$.
Here the equations $\mathcal E_f$: $u_{xy}=f(x,y,u,u_x,u_y)$ and 
\smash{$\mathcal E_{\tilde f}$}: $\tilde u_{\tilde x\tilde y}=\tilde f(\tilde x,\tilde y,\tilde u,\tilde u_{\tilde x},\tilde u_{\tilde y})$ 
belong to the class~$\mathcal H_{\rm gen}$, 
where $f\in{\rm C}^\infty(\Omega,\mathbb R)$ and $\tilde f\in{\rm C}^\infty(\tilde\Omega,\mathbb R)$ 
with connected open domains $\Omega$ and~$\tilde\Omega$ contained in~$\mathbb R^5$.
$\Phi$ is a contact transformation with the independent variables $(x,y)$ and the dependent variable~$u$,
$\Phi$: $(\tilde x,\tilde y,\tilde u,\tilde u_{\tilde t},\tilde u_{\tilde y})=(X,Y,U,U^x,U^y)$,
that is defined by a diffeomorphism from~$\Omega$ onto~$\tilde\Omega$
and maps $\mathcal E_f$ to \smash{$\mathcal E_{\tilde f}$}, $\Phi_*\mathcal E_f=\smash{\mathcal E_{\tilde f}}$.
Thus, the functions $X$, $Y$, $U$, $U^x$ and~$U^y$ in the components of the transformation~$\Phi$ 
are smooth functions of $z:=(x,y,u,u_x,u_y)\in\Omega$ with $\mathrm J:=\big|\p(X,Y,U,U^x,U^y)/\p(x,y,u,u_x,u_y)\big|\ne0$
that satisfy the contact condition
\begin{gather}\label{eq:GenKGEqsContactCondition}
\begin{split}
&U^x\mathrm D_xX+U^y\mathrm D_xY=\mathrm D_xU,\\
&U^x\mathrm D_yX+U^y\mathrm D_yY=\mathrm D_yU.
\end{split}
\end{gather}
Collecting coefficients of the second derivatives of~$u$ in this condition leads to the system
\begin{gather}\label{eq:GenKGEqsSplitContactCondition}
\begin{split}
&U^xX_{u_x}+U^yY_{u_x}=U_{u_x},\quad U^x\hat{\mathrm D}_xX+U^y\hat{\mathrm D}_xY=\hat{\mathrm D}_xU,\\
&U^xX_{u_y}+U^yY_{u_y}=U_{u_y},\quad U^x\hat{\mathrm D}_yX+U^y\hat{\mathrm D}_yY=\hat{\mathrm D}_yU,
\end{split}
\end{gather}
where $\hat{\mathrm D}_x=\p_x+u_x\p_u$ and $\hat{\mathrm D}_y=\p_y+u_y\p_u$ 
are the truncated operators of total derivatives with respect to~$x$ and~$y$, respectively.

\begin{lemma}\label{lem:GenKGEqsContactAdmTransDetEqs}
Up to the discrete equivalence transformation 
$\mathscr I^0$: $\tilde x=y$, $\tilde y=x$, $\tilde u=u$, $\tilde u_{\tilde x}=u_y$, $\tilde u_{\tilde y}=u_x$, $\tilde f=f$
of the class~$\mathcal H_{\rm gen}$, 
the components of the transformational part~$\Phi$ of any contact admissible transformation~$\mathcal T$ in this class
satisfy, on the entire corresponding domain~$\Omega$, the joint system of~\eqref{eq:GenKGEqsSplitContactCondition} and
\begin{subequations}\label{eq:GenKGEqsConditionForContactAdmTransReduced}
\begin{gather}\label{eq:GenKGEqsConditionForContactAdmTransReducedA}
\hat{\mathrm D}_yX+fX_{u_x}=0,\quad X_{u_y}=0,\quad U^x_{u_y}=(\Phi^*\tilde f)Y_{u_y},\quad \hat{\mathrm D}_yU^x+fU^x_{u_x}=(\Phi^*\tilde f)\hat{\mathrm D}_yY,
\\\label{eq:GenKGEqsConditionForContactAdmTransReducedB}
\hat{\mathrm D}_xY+fY_{u_y}=0,\quad Y_{u_x}=0,\quad U^y_{u_x}=(\Phi^*\tilde f)X_{u_x},\quad \hat{\mathrm D}_xU^y+fU^y_{u_y}=(\Phi^*\tilde f)\hat{\mathrm D}_xX.
\end{gather}
\end{subequations}
\end{lemma}

\begin{proof}
We expand the condition $\Phi_*\mathcal E_f=\smash{\mathcal E_{\tilde f}}$
via substituting the expression for the involved jet variables with tildes 
in terms of the jet variables without tildes,
\begin{gather}\label{eq:GenKGEqsConditionForContactAdmTrans}\arraycolsep=.5ex
\left|\begin{array}{cc}\mathrm D_xU^y&\mathrm D_yU^y\\\mathrm D_xY  &\mathrm D_yY  \end{array}\right|=
\left|\begin{array}{cc}\mathrm D_xX  &\mathrm D_yX  \\\mathrm D_xU^x&\mathrm D_yU^x\end{array}\right|=(\Phi^*\tilde f)
\left|\begin{array}{cc}\mathrm D_xX  &\mathrm D_yX  \\\mathrm D_xY  &\mathrm D_yY  \end{array}\right|
\quad\mbox{on solutions of}\ \mathcal E_f.
\end{gather}
Here $\Phi^*$ denotes the pullback by~$\Phi$, $\Phi^*\tilde f:=\tilde f(X,Y,U,U^x,U^y)$.
The first equality in~\eqref{eq:GenKGEqsConditionForContactAdmTrans}
is a differential consequence of the system~\eqref{eq:GenKGEqsContactCondition}.
Define two $2\times4$ smooth matrix functions on~$\Omega$,
\[
A_1:=\left(
\begin{array}{c}
\mathfrak QX\\ \mathfrak QU^x-(\Phi^*\tilde f)\mathfrak QY
\end{array} 
\right)_{\mathfrak Q\in\mathcal M} 
\quad\mbox{and}\quad
A_2:=\left(
\begin{array}{c}
\mathfrak QY\\ \mathfrak QU^y-(\Phi^*\tilde f)\mathfrak QX
\end{array} 
\right)_{\mathfrak Q\in\mathcal M}, 
\]
where $\mathcal M:=(\p_{u_x},\,\p_{u_y},\,\hat{\mathrm D}_x+f\p_{u_y},\,\hat{\mathrm D}_y+f\p_{u_x})$.

It is obvious that the set $\Omega_1:=\{z\in\Omega\mid\rank A_1(z)=\rank A_2(z)=2\}$ is open.
Moreover, it is dense in~$\Omega$. 
Indeed, $\Omega_1=\Omega\setminus(\Omega_{01}\cup\Omega_{02}\cup\Omega_{03}\cup\Omega_{04})$, 
where 
\begin{gather*}
\Omega_{01}:=\big\{z\in\Omega\mid\mathfrak QX(z)=0,\,\mathfrak Q\in\mathcal M\big\},\quad 
\Omega_{03}:=\big\{z\in\Omega\setminus\Omega_{01}\mid\rank A_1=1\big\},\quad 
\\
\Omega_{02}:=\big\{z\in\Omega\mid\mathfrak QY(z)=0,\,\mathfrak Q\in\mathcal M\big\},\quad 
\Omega_{04}:=\big\{z\in\Omega\setminus\Omega_{02}\mid\rank A_2=1\big\}.
\end{gather*}
If the set~$\Omega_{01}$ (resp.\ $\Omega_{02}$) contains an open subset~$\mathcal U$, 
then the row of~$\mathrm J$ that is associated with~$X$ (resp.\ $Y$) and thus the Jacobian~$\mathrm J$ itself
would vanish on~$\mathcal U$, which is a contradiction. 
Therefore, the subsets~$\Omega_{01}$ and $\Omega_{02}$ are nowhere dense.  
Suppose that the set~$\Omega_{03}$ contains an open subset~$\mathcal U$. 
Then there exists $\Lambda\in{\rm C}^\infty(\mathcal U,\mathbb R)$ 
such that $\mathfrak QU^x-(\Phi^*\tilde f)\mathfrak QY=\Lambda\mathfrak QX$ on~$\mathcal U$ 
for any $\mathfrak Q\in\mathcal M$, 
which implies the equations 
$\mathrm D_xU^x=\Lambda\mathrm D_xX+(\Phi^*\tilde f)\mathrm D_xY$ and 
$\mathrm D_yU^x=\Lambda\mathrm D_yX+(\Phi^*\tilde f)\mathrm D_yY$ 
on~$\mathcal U$. 
We respectively subtract these equations from the equalities 
$\mathrm D_xU^x=U^{xx}\mathrm D_xX+U^{xy}\mathrm D_xY$ and
$\mathrm D_yU^x=U^{xx}\mathrm D_yX+U^{xy}\mathrm D_yY$, 
where 
$U^{xx}:=\Phi_{(2)}{\!}^*\tilde u_{\tilde x\tilde x}$, 
$U^{xy}:=\Phi_{(2)}{\!}^*\tilde u_{\tilde x\tilde y}$, 
and $\Phi_{(2)}$ denotes the prolongation of the contact transformation~$\Phi$ to the jet space~$\mathrm J^2(\mathbb R^2_{x,y},\mathbb R^{}_u)$. 
As a result, we derive the equations 
\begin{gather}\label{eq:GenKGEqsDerivedConditionForAdmTrans1}
(U^{xx}-\Lambda)\mathrm D_xX+(U^{xy}-\Phi^*\tilde f)\mathrm D_xY=0,\quad
(U^{xx}-\Lambda)\mathrm D_yX+(U^{xy}-\Phi^*\tilde f)\mathrm D_yY=0
\end{gather}
on~$\mathcal U$.
Restricting the equation~\eqref{eq:GenKGEqsDerivedConditionForAdmTrans1} on the manifold defined by~$\mathcal E_f$ in~$\mathrm J^2(\mathbb R^2_{x,y},\mathbb R^{}_u)$, 
where $u_{xy}=f$ and $U^{xy}=\Phi^*\tilde f$, leads to the equalities  
\[
\big(U^{xx}-\Lambda\big)(\hat{\mathrm D}_xX+fX_{u_y}+X_{u_x}u_{xx})=0,\quad
\big(U^{xx}-\Lambda\big)(\hat{\mathrm D}_yX+fX_{u_x}+X_{u_y}u_{yy})=0.
\]
Since the derivative~$\tilde u_{\tilde x\tilde x}$ is not constrained on the solution set of~\smash{$\mathcal E_{\tilde f}$}, 
this implies the equations $\hat{\mathrm D}_xX+fX_{u_y}+X_{u_x}u_{xx}=0$ and $\hat{\mathrm D}_yX+fX_{u_x}+X_{u_y}u_{yy}=0$, 
which can be split with respect to~$u_{xx}$ and~$u_{yy}$ into the condition $\mathfrak QX(z)=0$, $\mathfrak Q\in\mathcal M$, on~$\mathcal U$,
contradicting the definition of~$\Omega_{03}$. 
Therefore, the subset~$\Omega_{03}$ (and, similarly, $\Omega_{04}$) is nowhere dense.  

The splitting of the equation~\eqref{eq:GenKGEqsConditionForContactAdmTrans} with respect to~$u_{xx}$ and~$u_{yy}$
leads to the condition that the $2\times2$ minors of~$A_1$ and~$A_2$ 
associated with the pairs $(\mathfrak Q_1,\mathfrak Q_2)$ from the set $\mathcal M_1\times\mathcal M_2$, where 
\[
\mathcal M_1:=\big\{\hat{\mathrm D}_x+f\p_{u_y},\,\p_{u_x}\big\},\quad 
\mathcal M_2:=\big\{\hat{\mathrm D}_y+f\p_{u_x},\,\p_{u_y}\big\}, 
\]
vanish.
This condition implies, in view of the condition $\rank A_1(z)=\rank A_2(z)=2$ on~$\Omega_1$, the system 
\begin{gather*}
(\mathfrak Q_1X)(\mathfrak Q_2X)=0,\quad 
(\mathfrak Q_1X)\big(\mathfrak Q_2U^x-(\Phi^*\tilde f)\mathfrak Q_2Y\big)=0,\quad 
\big(\mathfrak Q_1U^x-(\Phi^*\tilde f)\mathfrak Q_1Y\big)(\mathfrak Q_2X)=0,\\
(\mathfrak Q_1Y)(\mathfrak Q_2Y)=0,\quad 
(\mathfrak Q_1Y)\big(\mathfrak Q_2U^y-(\Phi^*\tilde f)\mathfrak Q_2X\big)=0,\quad 
\big(\mathfrak Q_1U^y-(\Phi^*\tilde f)\mathfrak Q_1X\big)(\mathfrak Q_2Y)=0
\end{gather*}
on~$\Omega_1$ and thus, by continuity, on~$\Omega$.
Here the operator $\mathfrak Q_i$ runs through the set~$\mathcal M_i$, $i=1,2$.
Therefore, in each point of~$\Omega_1$, we have 
either $(\mathfrak Q_1X)_{\mathfrak Q_1\in\mathcal M_1}\ne(0,0)$ or $(\mathfrak Q_2X)_{\mathfrak Q_2\in\mathcal M_2}\ne(0,0)$ and
either $(\mathfrak Q_1Y)_{\mathfrak Q_1\in\mathcal M_1}\ne(0,0)$ or $(\mathfrak Q_2Y)_{\mathfrak Q_2\in\mathcal M_2}\ne(0,0)$.

If $(\mathfrak Q_1X)_{\mathfrak Q_1\in\mathcal M_1}\ne(0,0)$ at a point $z_0\in\Omega_1$, 
then $(\mathfrak Q_2Y)_{\mathfrak Q_2\in\mathcal M_2}\ne(0,0)$ at~$z_0$. 
Indeed, otherwise $(\mathfrak Q_1X)_{\mathfrak Q_1\in\mathcal M_1}\ne(0,0)$ 
and $(\mathfrak Q_1Y)_{\mathfrak Q_1\in\mathcal M_1}\ne(0,0)$ simultaneously at~$z_0$ and thus on some neighborhood~$\mathcal U$ of~$z_0$ in~$\Omega_1$. 
This implies the equations $\mathfrak Q_2X=\mathfrak Q_2Y=0$ and then $\mathfrak Q_2U^x=\mathfrak Q_2U^y=0$ with $\mathfrak Q_2\in\mathcal M_2$. 
In particular, $X_{u_y}=Y_{u_y}=U^x_{u_y}=U^y_{u_y}=0$
and, in view of the third equation in~\eqref{eq:GenKGEqsSplitContactCondition}, $U_{u_y}=0$, 
which contradicts the nondegeneracy of~$\Phi$. 
It is obvious that the inverse implication, $(\mathfrak Q_1X)_{\mathfrak Q_1\in\mathcal M_1}\ne(0,0)$ at a point of~$\Omega_1$ 
if $(\mathfrak Q_2Y)_{\mathfrak Q_2\in\mathcal M_2}\ne(0,0)$ at the same point, also holds true. 
Analogously, we also have that $(\mathfrak Q_2X)_{\mathfrak Q_2\in\mathcal M_2}\ne(0,0)$ at a point $z_0\in\Omega_1$ 
if and only if $(\mathfrak Q_1Y)_{\mathfrak Q_1\in\mathcal M_1}\ne(0,0)$ at~$z_0$.
Denoting 
\begin{gather*}
\Xi_1:=\big\{z\in\Omega_1\mid(\mathfrak Q_1X)_{\mathfrak Q_1\in\mathcal M_1}\ne(0,0),\,(\mathfrak Q_2Y)_{\mathfrak Q_2\in\mathcal M_2}\ne(0,0)\big\},\\
\Xi_2:=\big\{z\in\Omega_1\mid(\mathfrak Q_2X)_{\mathfrak Q_2\in\mathcal M_2}\ne(0,0),\,(\mathfrak Q_1Y)_{\mathfrak Q_1\in\mathcal M_1}\ne(0,0)\big\},
\end{gather*}
we get $\Omega_1=\Xi_1\sqcup\Xi_2$. 
On the closure $\mathop{\rm cl}(\Xi_1)$ of~$\Xi_1$ in~$\Omega$, the components of~$\Phi$ satisfy 
the system~\eqref{eq:GenKGEqsConditionForContactAdmTransReduced}.
Note that on the subset, where $X_{u_x}\ne0$ or $Y_{u_y}\ne0$,
the last equations in~\eqref{eq:GenKGEqsConditionForContactAdmTransReducedA} and in~\eqref{eq:GenKGEqsConditionForContactAdmTransReducedB}
are differential consequences of the other equations in~\eqref{eq:GenKGEqsConditionForContactAdmTransReduced} 
supplemented with the equations~\eqref{eq:GenKGEqsSplitContactCondition}. 
On the closure $\mathop{\rm cl}(\Xi_2)$ of~$\Xi_2$ in~$\Omega$, the components of~$\Phi$ satisfy the analogous system 
obtained from the system~\eqref{eq:GenKGEqsConditionForContactAdmTransReduced} by permutation of $(x,u_x)$ and $(y,u_y)$. 
The sets~$\mathop{\rm cl}(\Xi_1)$ and~$\mathop{\rm cl}(\Xi_2)$ are disjoint since otherwise $\mathrm J=0$ on $\mathop{\rm cl}(\Xi_1)\cap\mathop{\rm cl}(\Xi_2)$.
Therefore, $\Omega=\mathop{\rm cl}(\Omega_1)=\mathop{\rm cl}(\Xi_1\sqcup\Xi_2)=\mathop{\rm cl}(\Xi_1)\sqcup\mathop{\rm cl}(\Xi_2)$, 
which implies that both the sets~$\mathop{\rm cl}(\Xi_1)$ and~$\mathop{\rm cl}(\Xi_2)$ are open and closed in~$\Omega$, 
and hence one of the sets is empty and the other coincides with~$\Omega$. 
Up to the discrete equivalence transformation $\mathscr I^0$, 
we can assume that $\mathop{\rm cl}(\Xi_1)=\Omega$ and $\mathop{\rm cl}(\Xi_2)=\varnothing$, 
i.e., the system~\eqref{eq:GenKGEqsConditionForContactAdmTransReduced} holds on the entire set~$\Omega$.
\looseness=-1
\end{proof}

\section{Description of equivalence groupoids}
\label{sec:GenKGEqsEquivGroupoids}

We first compute the point equivalence groupoid of the class~$\mathcal H_{\rm gen}$ 
and then single out its three disjoint subclasses, 
$\mathcal H_{xy}:=\mathcal H_x\cap\mathcal H_y$,
$\mathcal H_x':=\mathcal H_x\setminus\mathcal H_y$ and
$\mathcal H_y':=\mathcal H_y\setminus\mathcal H_x$,  
where 
\begin{gather*}
\mathcal H_x:=\big\{\mathcal E_f\in\mathcal H_{\rm gen}\mid f=F^0(x,y,u,u_y)+F^1(x,y,u,u_y)u_x,\,F^1_x+F^0F^1_{u_y}=F^0_u+F^1F^0_{u_y}\big\},\\
\mathcal H_y:=\big\{\mathcal E_f\in\mathcal H_{\rm gen}\mid f=F^0(x,y,u,u_x)+F^1(x,y,u,u_x)u_y,\,F^1_y+F^0F^1_{u_x}=F^0_u+F^1F^0_{u_x}\big\},
\end{gather*}
and hence 
\begin{align*}
\mathcal H_{xy}=\big\{&\mathcal E_f\in\mathcal H_{\rm gen}\mid f=f^0(x,y,u)+f^1(x,y,u)u_x+f^2(x,y,u)u_y+f^3(x,y,u)u_xu_y,\\
                   &f^3_y=f^1_u,\,f^3_x=f^2_u,\,f^2_y=f^1_x=f^0_u+f^1f^2-f^3f^0\big\}.
\end{align*}
We will prove that any contact admissible transformation in~$\mathcal H_{\rm gen}$
that is not the first-order prolongation of a point admissible transformation in~$\mathcal H_{\rm gen}$
belongs to the contact equivalence groupoid of $\mathcal H_x\cup\mathcal H_y$. 
It is obvious that the subclasses~$\mathcal H_x$ and~$\mathcal H_y$ (resp.\ the subclasses~$\mathcal H_x'$ and~$\mathcal H_y'$)
are mapped onto each other by the permutation~$\mathscr I^0$. 
This is why any assertion that is true for one of these two subclasses holds true for the other after the modification using~$\mathscr I^0$, 
and each of them can be substituted by the other in the assertions below. 
Any equation~$\mathcal E_f$ from the subclass~$\mathcal H_x$ (resp.\ $\mathcal H_y$ or~$\mathcal H_{xy}$) 
can be (locally) represented in the form $\mathrm D_xg=0$ (resp.\ $\mathrm D_yh=0$ or~$\mathrm D_x\mathrm D_y\theta=0$). 
Here $g=g(x,y,u,u_y)$, $h=h(x,y,u,u_x)$ and~$\theta=\theta(x,y,u)$ 
are arbitrary smooth functions of their arguments with $g_{u_y}\ne0$, $h_{u_x}\ne0$ and $\theta_u\ne0$. 
(The representation for equations from the subclass~$\mathcal H_{xy}$ 
follows from those for the subclasses~$\mathcal H_x$ and~$\mathcal H_y$
in view of the theorem on null divergences \cite[Theorem 4.24]{olve1993A}.)
This is why all the equations from the subclasses~$\mathcal H_x$ and~$\mathcal H_y$, 
not to mention their intersection $\mathcal H_{xy}$, are Darboux integrable in a trivial way. 
The reparameterized class~$\mathcal H_y$, where $h$ is assumed as the arbitrary element instead of~$f$, 
possesses gauge equivalence transformations (see  the definition of this notion in \cite{popo2010a})
\[(\tilde x,\tilde y,\tilde u,\tilde u_{\tilde x},\tilde u_{\tilde y})=(x,y,u,u_x,u_y),\quad \tilde h=H(x,h),\] 
where $H$ is an arbitrary smooth function of~$(x,h)$ with $H_h\ne0$.
 
\begin{theorem}\label{thm:GenKGEqsEquivGroupoid}
{\samepage\begin{subequations}\label{eq:GenKGEqsGsim}
(i) The class~$\mathcal H_{\rm gen}$ is normalized in the usual sense with respect to point transformations.
Its point equivalence group~$G^\sim_{\rm gen}$ coincides with its contact equivalence group and 
is generated by 
the discrete equivalence transformation $\mathscr I^0$: $\tilde x=y$, $\tilde y=x$, $\tilde u=u$, $\tilde u_{\tilde x}=u_y$, $\tilde u_{\tilde y}=u_x$, $\tilde f=f$
and the transformations of the form
\begin{gather}\label{eq:GenKGEqsGsimA}
\tilde x=X(x),\quad \tilde y=Y(y),\quad \tilde u=U(x,y,u),\quad 
\tilde u_{\tilde x}=\frac{U_x+U_uu_x}{X_x},\quad \tilde u_{\tilde y}=\frac{U_y+U_uu_y}{Y_y},
\\ \label{eq:GenKGEqsGsimB}
\tilde f=\frac1{X_xY_y}(U_uf+U_{xy}+U_{xu}u_y+U_{yu}u_x+U_{uu}u_xu_y),
\end{gather}
\end{subequations}
where $X$, $Y$ and $U$ are arbitrary smooth functions of their arguments with $X_xY_yU_u\ne0$.

}

\medskip

(ii) The contact equivalence groupoid~$\mathcal G^{\sim\rm c}_{\rm gen}$ of the class~$\mathcal H_{\rm gen}$ 
is generated by 
\begin{itemize}\itemsep=0ex
\item
the first-order prolongation of the point equivalence groupoid~$\mathcal G^{\sim\rm p}_{\rm gen}$ of this class,
\item
the contact vertex group~$\mathcal G^{\rm c}_0$ of the wave equation~$\mathcal E_0\colon$ $u_{xy}=0$ and 
\item
the subgroupoid~$\mathcal G^{\sim\rm q}_y$ of the equivalence groupoid~$\mathcal G^{\sim\rm c}_y$ of the subclass~$\mathcal H_y'$ 
whose elements are induced, modulo gauge equivalence within the reparameterized class~$\mathcal H_y'$, 
by the pullbacks of the transformations of the form
$\tilde\tau=\eta$, $\tilde\xi=\xi$, $\tilde\upsilon=\Upsilon(\tau,\xi,\upsilon,\eta)$, $\tilde\eta=\tau$ with $\Upsilon_\upsilon\ne0$
in the space with the coordinates $(\tau,\xi,\upsilon,\eta)$ 
by the mapping $\Psi$: $(\tau,\xi,\upsilon)=(x,y,u)$, $\eta=h(x,y,u,u_x)$,
where $\tilde u_{\tilde x}=\Psi^*\Upsilon_\eta$, $\tilde u_{\tilde y}=\Psi^*\Upsilon_\xi+u_y\Psi^*\Upsilon_\upsilon$,
and $h$ is an arbitrary element of the reparameterized class~$\mathcal H_y'$.
\end{itemize}
The subclasses 
$\mathcal H_{xy}$, 
$\mathcal C_0:=\mathcal H_x\bigtriangleup\mathcal H_y=\mathcal H_x'\cup\mathcal H_y'$ and 
$\mathcal C_1:=\mathcal H_{\rm gen}\setminus(\mathcal H_x\cup\mathcal H_y)$ are invariant under the action of~$\mathcal G^{\sim\rm c}_{\rm gen}$. 
In other words, the partition of the class~$\mathcal H_{\rm gen}$ into these three subclasses, 
$\mathcal H_{\rm gen}=\mathcal H_{xy}\sqcup\mathcal C_0\sqcup\mathcal C_1$, 
induces the corresponding partitions of its contact equivalence groupoid, 
$\mathcal G^{\sim\rm c}_{\rm gen}=\mathcal G^{\sim\rm c}_{\mathcal H_{xy}}\sqcup\mathcal G^{\sim\rm c}_{\mathcal C_0}\sqcup\mathcal G^{\sim\rm c}_{\mathcal C_1}$.

\medskip

(iii) The $\mathcal G^{\sim\rm c}_{\rm gen}$-orbit of the wave equation~$\mathcal E_0$ coincides with the subclass~$\mathcal H_{xy}$.

\medskip

\begin{subequations}\label{eq:GenKGEqsWaveEqContactSymGroup}
(iv) The contact symmetry group~$G^{\rm c}_0$ of~$\mathcal E_0$ is generated by the discrete permutation transformation 
$\tilde x=y$, $\tilde y=x$, $\tilde u=u$, $\tilde u_{\tilde x}=u_y$, $\tilde u_{\tilde y}=u_x$
and the contact transformations of the~form 
\begin{gather}\label{eq:GenKGEqsWaveEqContactSymGroupA}
\tilde x=X(x,u_x),\quad \tilde y=Y(y,u_y),\quad \tilde u=cu+U^1(x,u_x)+U^2(y,u_y),
\\\label{eq:GenKGEqsWaveEqContactSymGroupB}
\tilde u_{\tilde x}=U^x(x,u_x),\quad 
\tilde u_{\tilde y}=U^y(y,u_y),
\end{gather}
\end{subequations}
where $X$, $Y$, $U^1$ and $U^2$ are arbitrary smooth functions of their arguments 
and $c$ is an arbitrary constant with $(X_x,X_{u_x})\ne(0,0)$, $(Y_y,Y_{u_y})\ne(0,0)$, $c\ne0$ and
\begin{gather}\label{eq:GenKGEqsWaveEqContactSymGroupConditions}
U^1_{u_x}=U^xX_{u_x},\quad cu_x+U^1_x=U^xX_x,\quad
U^2_{u_y}=U^yY_{u_y},\quad cu_y+U^2_y=U^yY_y.
\end{gather}
\end{theorem}

\begin{proof}
\begin{subequations}\label{eq:GenKGEqsPointAdmTransDetEqs}
We sketch the proof using the statement, the notation and the proof of Lemma~\ref{lem:GenKGEqsContactAdmTransDetEqs} 
and the equivalence generated by the transformation~$\mathscr I^0$. 
Looking for contact equivalence transformations of the class~$\mathcal H_{\rm gen}$, 
we split, with respect to the source and target values of the arbitrary element, $f$ and~$\tilde f$,
the equations of the system~\eqref{eq:GenKGEqsConditionForContactAdmTransReduced}
that do not involve these values simultaneously. 
Taking into account the contact condition~\eqref{eq:GenKGEqsSplitContactCondition}, 
we obtain the system 
\begin{gather}\label{eq:GenKGEqsPointAdmTransDetEqsA}
X_{u_x}=X_{u_y}=X_u=X_y=0, \quad 
Y_{u_x}=Y_{u_y}=Y_u=Y_x=0, \quad 
U_{u_x}=U_{u_y}=0,
\\\label{eq:GenKGEqsPointAdmTransDetEqsB}
U^x=\frac{\hat{\mathrm D}_xU}{X_x},\quad
U^y=\frac{\hat{\mathrm D}_yU}{Y_y},\quad
\end{gather}
\end{subequations}
where $X_xY_yU_u\ne0$ since $\mathrm J\ne0$. 
In addition, we derive the equation~\eqref{eq:GenKGEqsGsimB} defining the $f$-com\-po\-nents of equivalence transformations. 
Any point admissible transformation in the class~$\mathcal H_{\rm gen}$ relates its source and target according to~\eqref{eq:GenKGEqsGsimB}, 
its transformational part satisfies the equations~\eqref{eq:GenKGEqsPointAdmTransDetEqsA} and thus 
it is induced by a point equivalence transformation of the class~$\mathcal H_{\rm gen}$.
This proves point~{\it(i)} of the theorem.

Under the constraint $X_{u_x}=Y_{u_y}=0$, 
the systems~\eqref{eq:GenKGEqsSplitContactCondition} and~\eqref{eq:GenKGEqsConditionForContactAdmTransReduced} 
imply the system~\eqref{eq:GenKGEqsPointAdmTransDetEqs}, 
i.e., any contact admissible transformation whose $x$- and $y$-components do not depend respectively on~$u_x$ and on~$u_y$ 
is the first-order prolongation of a point admissible transformation.

If both $f=0$ and $\tilde f=0$, 
then the joint system~\eqref{eq:GenKGEqsSplitContactCondition}--\eqref{eq:GenKGEqsConditionForContactAdmTransReduced} 
reduces to $X_u=X_y=0$, $Y_u=Y_x=0$, $U^x_{u_y}=U^x_u=U^x_y=0$, $U^y_{u_x}=U^y_u=U^y_x=0$, 
$U^xX_{u_x}=U_{u_x}$, $U^x\hat{\mathrm D}_xX=\hat{\mathrm D}_xU$, $U^yY_{u_y}=U_{u_y}$, $U^y\hat{\mathrm D}_yY=\hat{\mathrm D}_yU$, 
where $(X_x,X_{u_x})\ne(0,0)$, $(Y_y,Y_{u_y})\ne(0,0)$ and $U_u\ne0$ in view of $\mathrm J\ne0$,
which implies $U_{u_xu_y}=U_{u_xu}=U_{u_xy}=U_{u_yu}=U_{u_yx}=0$, and hence point~{\it(iv)} holds true. 

If just $f=0$, then similarly we derive $X_u=X_y=0$ and $Y_u=Y_x=0$, 
i.e., $\tilde x=X(x,u_x)$ and $\tilde y=Y(y,u_y)$ with $(X_x,X_{u_x})\ne(0,0)$ and $(Y_y,Y_{u_y})\ne(0,0)$. 
Factoring out the element of~$\mathcal G^{\rm c}_0$ with the same $x$- and $y$-components as~$\mathcal T$ 
and with the identity $u$-component, $\tilde u=u$, from~$\mathcal T$, 
we can set $X=x$ and $Y=y$, and then $U=U(x,y,u)$. 
This makes point~{\it(iii)} obvious. 
Moreover, it becomes clear that 
the subclass~$\mathcal H_{xy}$ is invariant under the action of~$\mathcal G^{\sim\rm c}_{\rm gen}$.  

If both $X_{u_x}\ne0$ and $Y_{u_y}\ne0$, then the first equations 
in~\eqref{eq:GenKGEqsConditionForContactAdmTransReducedA} and~\eqref{eq:GenKGEqsConditionForContactAdmTransReducedB} 
immediately imply $f_{u_xu_x}=f_{u_yu_y}=0$ and, after splitting them with respect to~$u_x$ and~$u_y$ 
and testing the obtained equations on the compatibility with respect to~$X$ and~$Y$, 
give the other constraints on~$f$, jointly meaning that $\mathcal E_f\in\mathcal H_{xy}$. 
In other words, the $(x,y)$-components of $\mathcal T\in\mathcal G^{\sim\rm c}_{\rm gen}$ 
simultaneously depend on the first derivatives of~$u$
only if $\mathcal T\in\mathcal G^{\sim\rm c}_{\mathcal H_{xy}}$.

Suppose that only one of the derivatives $X_{u_x}$ and $Y_{u_y}$ does not vanish. 
Up to the equivalence generated by~$\mathscr I^0$, we can assume that $X_{u_x}\ne0$ and $Y_{u_y}=0$.
Then we derive from the joint system~\eqref{eq:GenKGEqsSplitContactCondition}--\eqref{eq:GenKGEqsConditionForContactAdmTransReduced} 
that 
\begin{gather*}
X=X(x,y,u,u_x),\quad Y=Y(y),\quad U=U(x,y,u,u_x),\quad\mbox{where}\quad U_{u_x}\hat{\mathrm D}_xX=X_{u_x}\hat{\mathrm D}_xU,\\ 
U^x=\frac{U_{u_x}}{X_{u_x}},\quad U^y=\frac{\hat{\mathrm D}_yU-U^x\hat{\mathrm D}_yX}{Y_y},\quad
\tilde f=\frac{U^y_{u_x}}{X_{u_x}},\quad f=\frac{\hat{\mathrm D}_yX}{X_{u_x}}. 
\end{gather*}
The components of the inverse~$\mathcal T^{-1}$ of~$\mathcal T$ are of the same form as the respective components of~$\mathcal T$. 
In particular, the derivative of the $x$-component of~$\mathcal T^{-1}$ with respect to~$u_x$ does not vanish as well 
since otherwise the admissible transformation~$\mathcal T^{-1}$ and thus the admissible transformation~$\mathcal T$ 
are the first-order prolongations of point admissible transformations in the class~$\mathcal H_{\rm gen}$, 
which contradicts the condition $X_{u_x}\ne0$.   
This is why the first equation in~\eqref{eq:GenKGEqsConditionForContactAdmTransReducedA} implies 
that both the equations~$\mathcal E_f$ and~$\mathcal E_{\tilde f}$ belong to the class~$\mathcal H_y$. 
Together with point~{\it(i)} and the proven $\mathcal G^{\sim\rm c}_{\rm gen}$-invariance of~$\mathcal H_{xy}$, 
this implies the $\mathcal G^{\sim\rm c}_{\rm gen}$-invariance of~$\mathcal H_x\cup\mathcal H_y$ and, hence, 
of~$\mathcal C_0$ and of~$\mathcal C_1$.
We can set $Y=y$ up to the ${\rm s}$-$G^\sim_{\rm gen}$-equivalence on $\mathcal G^{\sim\rm c}_{\rm gen}$.
Then, taking~$X$ as a value of the arbitrary element~$h$ for the reparameterized form $\mathrm D_yh=0$ 
of the equation~$\mathcal E_f$, we obtain that $\mathcal T\in\mathcal G^{\sim\rm q}_y$. 
\end{proof}

\begin{remark}
We explain, reformulate and expand the points of Theorem~\ref{thm:GenKGEqsEquivGroupoid}.

\medskip 

(i) The normalization of the class~$\mathcal H_{\rm gen}$ in the usual sense with respect to point transformations 
means that its point equivalence groupoid~$\mathcal G^{\sim\rm p}_{\rm gen}$
coincides with the action groupoid~$\mathcal G^{G^\sim_{\rm gen}}$ of the group~$G^\sim_{\rm gen}$, 
i.e., any point admissible transformation within the class~$\mathcal H_{\rm gen}$ 
is induced by its point equivalence transformation; see definitions in \cite{opan2020d,popo2010a,vane2020b}.

\medskip 

(ii) Two equations from the class~$\mathcal C_1$ 
are related by a contact transformation if and only if this transformation is the first-order prolongation of a point transformation  
that is the projection of an element of the group~$G^\sim_{\rm gen}$ to the space with coordinates~$(x,y,u)$. 
In other words, the class~$\mathcal C_1$ 
is normalized in the usual sense with respect to point transformations,
its point equivalence group coincides with~$G^\sim_{\rm gen}$ and with its contact equivalence group, 
and its contact equivalence groupoid is the first-order prolongation of its point equivalence groupoid.

\medskip 

(iii) Let $G^{\sim\,u}_{\rm gen}$ denote the subgroup of~$G^\sim_{\rm gen}$ 
that is constituted by the transformations of the form~\eqref{eq:GenKGEqsGsim} with $X(x)=x$ and $Y(y)=y$. 
The $\mathcal G^{\sim\rm c}_{\rm gen}$-orbit $\mathcal H_{xy}$ of the equation~$\mathcal E_0$ 
coincides with its $G^\sim_{\rm gen}$-orbit and, more specifically, with its $G^{\sim\,u}_{\rm gen}$-orbit.

An equation~$\mathcal E_f$ from the class~$\mathcal H_{\rm gen}$ reduces to~$\mathcal E_0$ by a contact transformation 
if and only if it reduces to~$\mathcal E_0$ by an element of~$G^{\sim\,u}_{\rm gen}$. 
Then the corresponding value of the arbitrary element~$f$ is necessarily of the form 
$f=(\theta_{xy}+\theta_{xu}u_y+\theta_{yu}u_x+\theta_{uu}u_xu_y)/\theta_u$, 
where $\theta$ is a smooth function of~$(x,y,u)$ with $\theta_u\ne0$.

The contact equivalence groupoid of the class~$\mathcal H_{xy}$ is generated by 
the contact vertex group~$\mathcal G^{\rm c}_0$ of the wave equation~$\mathcal E_0$ 
and the restriction of the action groupoid of the group~$G^{\sim\,u}_{\rm gen}$ to~$\mathcal H_{xy}$. 

\medskip 

(iv) 
In view of~\eqref{eq:GenKGEqsWaveEqContactSymGroupConditions}, 
the $u_x$-component (resp.\ the $u_y$-component) of a contact symmetry transformation of~$\mathcal E_0$ 
can be locally expressed in terms of the $x$- and the $u$-components (resp.\ the $y$- and the $u$-components).
More specifically, 
$U^x=U^1_{u_x}/X_{u_x}$, 
$U^x=(cu_x+U^1_x)/X_x$, 
$U^y=U^2_{u_y}/Y_{u_y}$ and
$U^y=(cu_y+U^2_y)/Y_y$,
whenever $X_{u_x}\ne0$, $X_x\ne0$, $Y_{u_y}\ne0$ and $Y_y\ne0$, respectively.

Excluding $U^x$ and~$U^y$ from~\eqref{eq:GenKGEqsWaveEqContactSymGroupConditions}, 
we derive $U^1_{u_x}X_x=(cu_x+U^1_x)X_{u_x}$ and $U^2_{u_y}Y_y=(cu_y+U^2_y)Y_{u_y}$. 
Whenever $X_{u_x}\ne0$, the former equation gives the following expression for~$U^1$: 
\[
U^1(x,u_x)=-c\int_{t_0}^x\Theta^1\big(x',X(x,u_x)\big)\,{\rm d}x'+\varphi^1\big(X(x,u_x)\big),
\]
where~$\Theta^1$ is the inverse of the function~$X$ with respect to~$u_x$, $\Theta^1\big(x,X(x,u_x)\big)=u_x$, 
and $\varphi^1$ is an arbitrary smooth function of its single argument. 
Similarly, the later equation integrates to an expression for~$U^2$ whenever $Y_{u_y}\ne0$,
\[
U^2(y,u_y)=-c\int_{y_0}^y\Theta^2\big(y',Y(y,u_y)\big)\,{\rm d}y'+\varphi^2\big(Y(y,u_y)\big),
\]
where~$\Theta^2$ is the inverse of the function~$Y$ with respect to~$u_y$, $\Theta^2\big(y,Y(y,u_y)\big)=u_x$, 
and $\varphi^2$ is an arbitrary smooth function of its single argument.

For contact symmetry transformation~\eqref{eq:GenKGEqsWaveEqContactSymGroup} of~$\mathcal E_0$, 
the equation~\eqref{eq:GenKGEqsWaveEqContactSymGroupConditions} implies that 
$U^1_{u_x}=0$ if $X_{u_x}=0$ and $U^2_{u_y}=0$ if $Y_{u_y}=0$. 
Therefore, such a transformation with $X_{u_x}=Y_{u_y}=0$ is the first-order prolongation 
of a point symmetry transformation of~$\mathcal E_0$. 

The point symmetry group~$G_0$ of the wave equation~$\mathcal E_0$ 
is generated by the discrete permutation transformation $\tilde x=y$, $\tilde y=x$, $\tilde u=u$
and the point transformations of the form 
$\tilde x=X(x)$, $\tilde y=Y(y)$, $\tilde u=cu+U^1(x)+U^2(y)$,
where $X$, $Y$, $U^1$ and $U^2$ are arbitrary smooth functions of their arguments with $X_x\ne0$ and $Y_y\ne0$,
and $c$ is an arbitrary nonzero constant.
%
Therefore, a complete list of discrete point symmetry transformations of the equation~$\mathcal E_0$ 
that are independent up to combining with each other and with continuous point symmetry transformations of this equation
is exhausted by the $(x,y)$-permutation and the three transformations 
each of which alternates the sign of one of the variables~$x$, $y$ and~$u$.
The quotient group of the point symmetry group~$G_0$ of~$\mathcal E_0$ 
with respect to its identity component is isomorphic to the group $\mathbb Z_2\times\mathbb Z_2\times\mathbb Z_2\times\mathbb Z_2$.
\end{remark}

\bigskip\par\noindent{\bf Acknowledgments.}
The author thanks Vyacheslav Boyko, Michael Kunzinger and Dmytro Popovych for helpful discussions. 
The research was supported by the Austrian Science Fund (FWF), projects P25064 and P28770.

\end{document}